\documentclass[12pt]{elsarticle}

\usepackage[english]{babel}
\usepackage{amsthm,amsmath,amssymb}
\usepackage{graphicx}
\usepackage[small,bf,hang]{caption} 

\usepackage{booktabs} 

\usepackage{graphicx}
\usepackage{subfig}

\usepackage{tikz}

\usepackage{hyperref}

\journal{???}

\newtheorem{theorem}{Theorem}
\newtheorem{lemma}{Lemma}

\newtheorem{corollary}[theorem]{Corollary}
\newtheorem{proposition}[theorem]{Proposition}

\usepackage{color}

\graphicspath{{figures/}}

\begin{document}

\begin{frontmatter}

\title{Structural and computational results on platypus graphs}

\author[gent]{Jan Goedgebeur\fnref{fwo}}
\ead{Jan.Goedgebeur@UGent.be}

\author[gent]{Addie Neyt}
\ead{Addie.Neyt@UGent.be}

\author[gent]{Carol T.\ Zamfirescu\fnref{fwo}}
\ead{czamfirescu@gmail.com}

\address[gent]{Department of Applied Mathematics, Computer Science \& Statistics\\
  Ghent University\\
Krijgslaan 281-S9, \\9000 Ghent, Belgium\\ }

\fntext[fwo]{Supported by a Postdoctoral Fellowship of the Research Foundation Flanders (FWO).}

\begin{abstract}
A \emph{platypus graph} is a non-hamiltonian graph for which every vertex-deleted subgraph is traceable. They are closely related to families of graphs satisfying interesting conditions regarding longest paths and longest cycles, for instance hypohamiltonian, leaf-stable, and maximally non-hamiltonian graphs.

In this paper, we first investigate cubic platypus graphs, covering all orders for which such graphs exist: in the general and polyhedral case as well as for snarks.
We then present (not necessarily cubic) platypus graphs of girth up to 16---whereas no hypohamiltonian graphs of girth greater than 7 are known---and study their maximum degree, generalising two theorems of Chartrand, Gould, and Kapoor.
Using computational methods, we determine the complete list of all non-isomorphic platypus graphs for various orders and girths.
Finally, we address two questions raised by the third author in [\emph{J. Graph Theory} \textbf{86} (2017) 223--243].
\end{abstract}


\end{frontmatter}

\section{Introduction}

\noindent A graph is \emph{hamiltonian} (\emph{traceable}) if it contains a spanning cycle (spanning path), which is called a \emph{hamiltonian cycle} (\emph{hamiltonian path}). A \emph{platypus graph}---or \emph{platypus} for short---is a non-hamiltonian graph in which every vertex-deleted subgraph is traceable~\cite{Za17}. We shall only consider graphs with at least three vertices and thus ignore the trivial platypus $K_2$. A graph is \emph{hypohamiltonian} (\emph{hypotraceable}) if the graph itself is non-hamiltonian (non-traceable), yet every vertex-deleted subgraph is hamiltonian (traceable). Every hypohamiltonian or hypotraceable graph is a platypus, but not vice-versa. For a survey on hypohamiltonicity, see Holton and Sheehan's~\cite{HS93}. For further results on platypuses and their relationship to other classes of graphs such as maximally non-hamiltonian or leaf-stable graphs, see~\cite{Za17}. For a graph $G$, we denote by $g(G)$ its \emph{girth}, i.e.\ the length of a shortest cycle in $G$. For $S \subset V(G)$ we write $G[S]$ for the subgraph of $G$ induced by $S$. We call a graph \emph{cubic} or \emph{$3$-regular} if all of its vertices are cubic, i.e.\ of degree~3.

This article is organised as follows. In Section~\ref{sect:cubic} we discuss cubic platypuses. Starting from Petersen's graph and Tietze's graph, we construct an infinite family of cubic platypuses, characterising all orders for which such graphs exist. We then investigate cubic \emph{polyhedral}---i.e.\ planar and 3-connected---platypuses, again characterising all orders for which such graphs exist. It turns out that the famous Lederberg-Bos\'ak-Barnette graph and the five other structurally similar smallest counterexamples to Tait's hamiltonian graph conjecture are the smallest cubic polyhedral platypuses. Finally, we also determine all orders for which platypus snarks exist.

In Section~\ref{sect:general} we go on to study the girth of (not necessarily cubic) platypuses, using a construction based on generalised Petersen graphs. Although no hypohamiltonian graph of girth greater than~7 is known, Coxeter's graph being the smallest hypohamiltonian graph of girth~7 (see~\cite{GZ17}), we find platypuses of all girths up to and including 16. In Section~\ref{sect:maxdeg}, we generalise two results of Chartrand, Gould, and Kapoor~\cite{CGK79}: we show that for a given $n$-vertex platypus, (i) its maximum degree is at most $n - 4$, and (ii) its set of vertices of degree $n-4$ induces a complete graph. Note that, combined with a previous result of the third author~\cite{Za17}, (i) implies that a maximally non-hamiltonian graph of order $n$ cannot have maximum degree $n-2$ or $n-3$. In Section~\ref{sect:computations} we present the complete lists of platypuses for various orders and lower bounds on the girth we obtained using our generation algorithm for platypuses.

The paper ends with a discussion of Questions 1 and 4 from~\cite[Section~7]{Za17}: we answer the former negatively and address the latter by giving the first non-trivial lower bound and improving the upper bound for the order of the smallest polyhedral platypus, as well as determining the orders of the smallest polyhedral platypuses of girth 4 and 5. Finally, we show that there exists a polyhedral platypus of order $n$ for every $n \ge 21$, strengthening~\cite[Theorem 5.3]{Za17}.

\section{Cubic platypuses}
\label{sect:cubic}

\begin{lemma} \label{lemma:transformation}
Let $G$ be a platypus containing a triangle $v_1v_2v_3$ whose vertices are cubic. Adding new vertices $v'_1, v'_2$ to $V(G)$, we have that
$$T(G) = \left(V(G) \cup \{ v'_1, v'_2 \}, (E(G) \setminus \{ v_1v_3, v_2v_3\} ) \cup \{ v_1v'_1, v_2v'_2, v'_1v'_2, v'_1v_3, v'_2v_3 \} \right)$$
is a platypus as well. The transformation $T$ preserves $3$-regularity, planarity, and $3$-connectedness.
\end{lemma}

\begin{proof}
See Figure~\ref{fig:transformation} for an illustration of the transformation $T$. In the remainder of this proof we consider $G - \{ v_1v_3, v_2v_3\}$ as a subgraph of $T(G)$.

\begin{figure}[h!tb]
	\centering
	\includegraphics[width=0.5\textwidth]{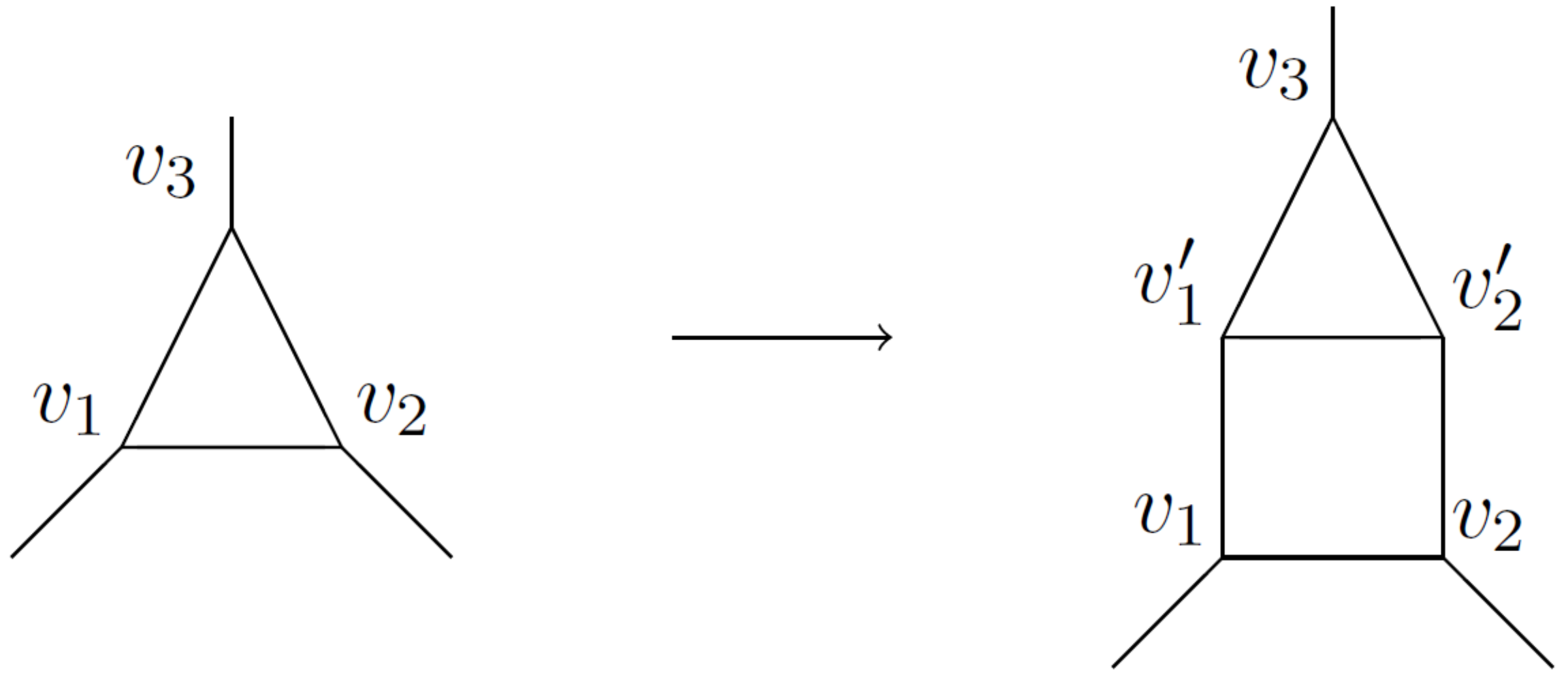}
	\caption{The transformation $T$.}
	\label{fig:transformation}
\end{figure}

Assume $T(G)$ contains a hamiltonian cycle ${\mathfrak h}$. Excluding symmetric cases, ${\mathfrak h} \cap T(G)[\{ v_1, v_2, v_3, v'_1, v'_2 \}]$ is $v_1v_2v'_2v'_1v_3$ or $v_1v'_1v_3v'_2v_2$. Replacing the former path with $v_1v_2v_3$ and the latter one with $v_1v_3v_2$ we obtain a hamiltonian cycle in $G$, a contradiction.

Let $v \in V(G)$. Since $G$ is a platypus, $G - v$ contains a hamiltonian path ${\mathfrak p}$. As $v_1$, $v_2$, and $v_3$ are cubic in $G$, the path ${\mathfrak p}$ contains at least one and at most two of the edges $v_1v_2, v_2v_3, v_3v_1$. There are six essentially different cases for ${\mathfrak p} \cap G[\{ v_1, v_2, v_3 \}]$: $v_1v_2$ (when $v = v_3$), $v_1v_2 + v_3$, $v_1v_3$ (when $v = v_2$), $v_1v_3 + v_2$, $v_1v_2v_3$, and $v_1v_3v_2$. In $T(G)$, replace these with $v_1v'_1v'_2v_2$, $v_1v'_1v'_2v_2 + v_3$, $v_1v'_1v'_2v_3$, $v_1v'_1v'_2v_3 + v_2$, $v_1v_2v'_2v'_1v_3$, and $v_1v'_1v_3v'_2v_2$, respectively. Thus, we obtain a hamiltonian path in $T(G) - v$ for every $v \in V(T(G)) \setminus \{ v'_1, v'_2 \}$.

Consider the hamiltonian path ${\mathfrak q}$ of $G - v_2$. Necessarily $v_1v_3 \in E({\mathfrak q})$. Let ${\mathfrak q}' = {\mathfrak q} - v_1v_3$ lie in $T(G)$. Now ${\mathfrak q}' \cup v_1v_2v'_2v_3$ is a hamiltonian path in $T(G) - v'_1$. For $v'_2$ the same idea is used.
\end{proof}

A graph $G$ is \emph{maximally non-hamiltonian} if $G$ is non-hamiltonian and for every pair of non-adjacent vertices $v$ and $w$, the graph $G$ contains a hamiltonian path between $v$ and $w$. We need the following.

\begin{lemma}[Zamfirescu~\cite{Za17}]\label{lemma:maxdeg}
Let $G$ be a maximally non-hamiltonian graph. Then $G$ is a platypus if and only if $\Delta(G) < |V(G)| - 1$.
\end{lemma}

We briefly expand on this: there exist infinitely many maximally non-hamiltonian graphs which are not platypuses (identify a vertex of $K_p$ with a vertex of $K_q$) and infinitely many platypuses which are not maximally non-hamiltonian (consider the dotted prisms defined in~\cite{Za17} and discussed in Section~\ref{sect:general}). However, by Lemma~\ref{lemma:maxdeg} every maximally non-hamiltonian graph $G$ of maximum degree less than $|V(G)| - 1$ \emph{is} a platypus.

\begin{theorem}
There is a cubic platypus of order $n$ if and only if $n$ is even and $n \ge 10$. The smallest cubic platypus is the Petersen graph.
\end{theorem}

\begin{proof} The graphs of Petersen and Tietze, depicted in Figure~\ref{fig:petersen_tietze}, are cubic maximally non-hamiltonian graphs~\cite{CE83} of order 10 and 12, respectively.

\begin{figure}[h!tb]
	\centering
	\includegraphics[width=0.5\textwidth]{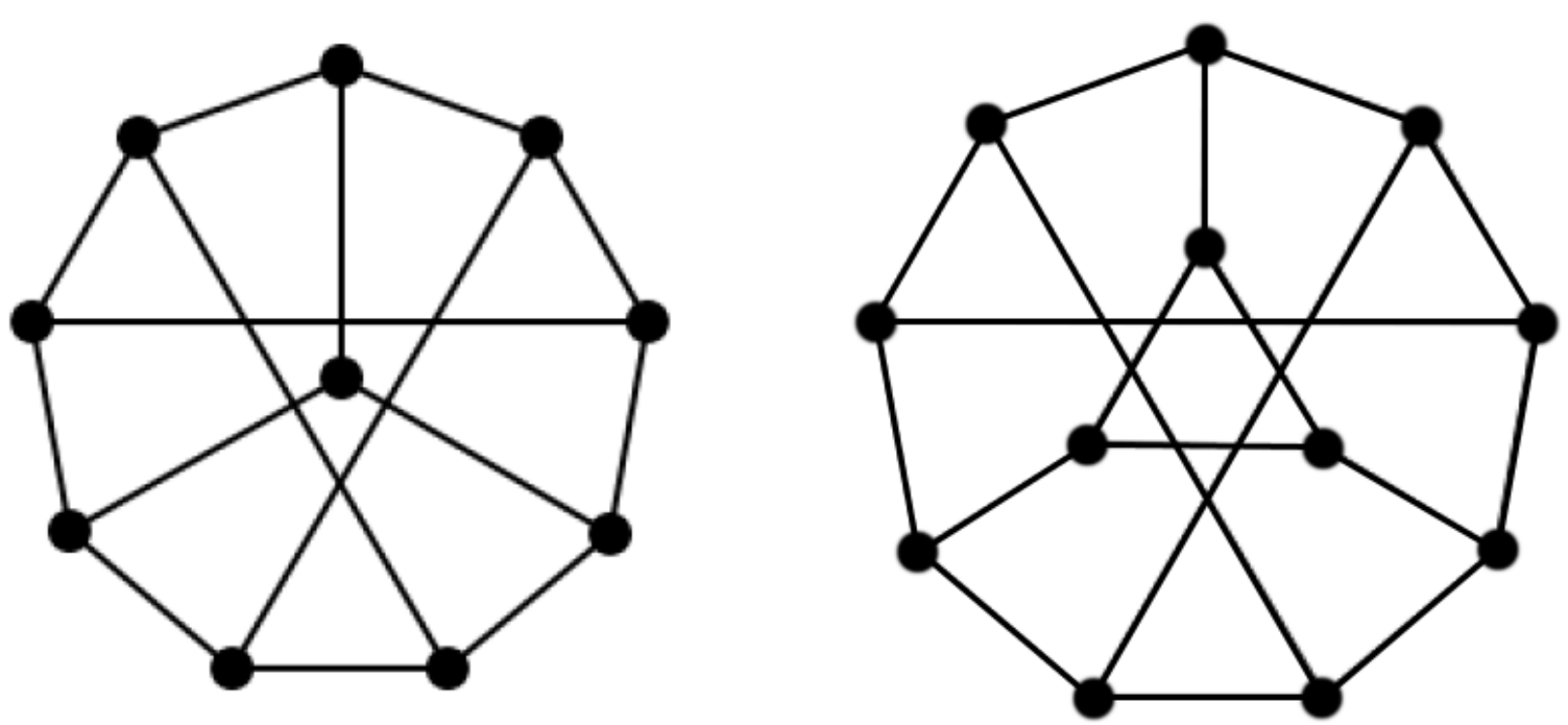}
	\caption{The Petersen graph (left-hand side) and the Tietze graph (right-hand side).}
	\label{fig:petersen_tietze}
\end{figure}

By Lemma~\ref{lemma:maxdeg} they are platypuses. Tietze's graph $G$ contains a triangle (see Figure~\ref{fig:petersen_tietze}), so we may apply transformation $T$ from Lemma~\ref{lemma:transformation}, obtaining a graph $T(G)$ that is a cubic platypus of order 14. Since $T(G)$ contains a triangle as well, we may iterate ad infinitum and obtain the infinite family of graphs $\{ T^k(G) \}_{k \ge 1}$ where $|V(T^k(G))| = 12 + 2k$. Thus, together with Petersen's graph and Tietze's graph, the first part of the theorem is shown.

The minimality of the Petersen graph follows from a result of Van Cleemput and the third author~\cite{Za17} (namely that there are no platypuses on fewer than 9 vertices---we ignore $K_2$), but can also be verified by inspecting Table~\ref{table:counts_platypuses} in Section~\ref{sect:computations}.
\end{proof}

Observe that we obtain the Tietze graph by replacing a vertex in the Petersen graph with a triangle~\cite{CE83}. One could wonder if we can always get a platypus in this manner, i.e.\ by replacing a cubic vertex with a triangle. In general, this is not the case: consider hypotraceable graphs. Whether in every \emph{traceable} platypus this transformation indeed yields again a platypus is unknown. 


\bigskip

We now turn our attention to cubic polyhedral platypuses.

\begin{theorem}\label{theorem:smallest_cubic_platypus}
There exists a cubic polyhedral platypus of order $n$ if and only if $n$ is even and $n \ge 38$. Furthermore, all six smallest non-hamiltonian cubic polyhedral graphs from~\cite{HM88}---amongst which the Lederberg-Bos\'ak-Barnette graph---are platypuses.
\end{theorem}

\begin{proof}
Holton and McKay~\cite{HM88} showed that the smallest non-hamiltonian cubic polyhedral graphs have 38 vertices and that there are exactly six such graphs of that order. Using two independent computer programs to test for platypusness, we verified the second part of the above statement. The programs are straightforward and verify that for every vertex $v$ of the input graph $G$, the vertex-deleted subgraph $G-v$ is traceable. Since these programs are very similar to the ones we used to test hypohamiltonicity in~\cite{GZ17}, we omit the details.

We now prove the first part of the theorem. Consider the Lederberg-Bos\'ak-Barnette graph $L$. We leave to the reader the simple proof of the fact that replacing a vertex $w$ in $L$ such that $L - w$ is hamiltonian with a triangle gives a cubic polyhedral platypus on 40~vertices, as the technique is very similar to what was used in the proof of Lemma~\ref{lemma:transformation}. The second author verified in her master's thesis~\cite{Ne17} that indeed such $w$ exist in $L$. We obtain a new graph $L'$. Since $L'$ contains a triangle, we may apply Lemma~\ref{lemma:transformation}. Iterating this, the proof is complete.
\end{proof}


The six graphs mentioned in Theorem~\ref{theorem:smallest_cubic_platypus} all have girth 4. In~\cite{HM88} it was also shown that the smallest non-hamiltonian cubic polyhedral graphs of girth 5 have 44 vertices and that there are exactly two such graphs of that order. Using two independent computer programs to test for platypusness, we verified the following.

\begin{theorem}
The smallest cubic polyhedral platypuses of girth $5$ have $44$ vertices and there are exactly two such graphs of that order (that is: the two smallest non-hamiltonian cubic polyhedral graphs of girth $5$ from~\cite{HM88}).
\end{theorem}


A particularly interesting subclass of cubic graphs is the class of \textit{snarks}: cubic cyclically 4-edge-connected graphs with chromatic index~4 (i.e.\ four colours are required in any proper edge-colouring) and girth at least~5. In~\cite{GZ18} the first and the third author showed the following.

\begin{theorem}[Theorem~3.6 in~\cite{GZ18}]\label{thm:snark_order_hypo}
There exists a hypohamiltonian snark of order $n$ if and only if $n \in \{ 10, 18, 20, 22 \}$ or $n$ is even and $n \ge 26$.
\end{theorem}

Recall that every hypohamiltonian graph is also a platypus. The only order for which there are snarks but no hypohamiltonian snarks is 24. By taking the complete lists of snarks from~\cite{BGHM13} and testing which ones are platypuses---using two independent computer programs---we showed that there exists a platypus snark of order $24$. In fact, we obtained the following stronger result.

\begin{proposition}
Every snark on up to $30$ vertices is a platypus.
\end{proposition}

Together with Theorem~\ref{thm:snark_order_hypo} this gives us:

\begin{theorem}
There exists a platypus snark of order $n$ if and only if $n = 10$ or $n$ is even and $n \ge 18$.
\end{theorem}

However, not every snark is a platypus; using a computer we showed:

\begin{proposition}
The smallest snark which is not a platypus has $32$ vertices. There are exactly thirteen snarks on $32$ vertices which are not a platypus. The most symmetric of these thirteen snarks has an automorphism group of size $16$ and is shown in Figure~\ref{fig:snark_no_platypus}.
\end{proposition}

\begin{figure}[h!tb]
	\centering
	\includegraphics[width=0.36\textwidth]{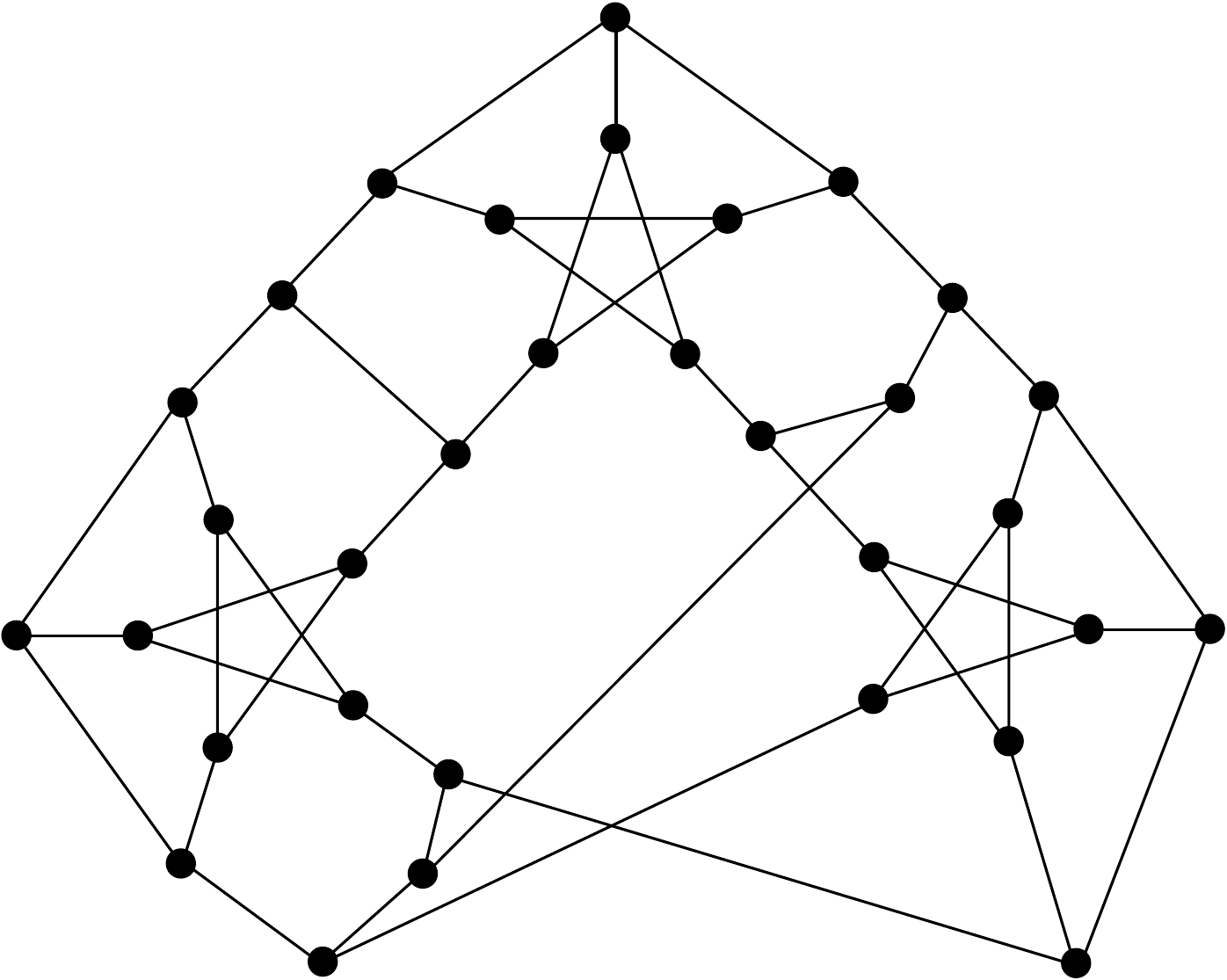}
	\caption{The most symmetric snark on 32 vertices which is not a platypus.} 
	\label{fig:snark_no_platypus}
\end{figure}

\section{Platypuses with a given girth}
\label{sect:general}
In this section we study the question for which integers $g \ge 3$ there exist platypuses of girth~$g$. It is known that there are hypohamiltonian graphs of girth $g$ whenever $3 \le g \le 7$. The order of the smallest hypohamiltonian graphs of a particular girth are given in~\cite{GZ17}. While infinitely many hypohamiltonian graphs of girth~7 are known~\cite{MS11}, the smallest among them being Coxeter's graph, no hypohamiltonian graphs of girth greater than 7 are known. A related open question in this direction was raised by M\'{a}\v{c}ajov\'{a} and \v{S}koviera~\cite{MS11}: Do infinitely many hypohamiltonian cubic graphs exist with both cyclic connectivity and girth $7$?

Can we exceed girth 7 in the (larger) family of platypuses? In order to address this question, we require the following well-known family of graphs. The {\em generalised Petersen graphs} ${\rm GP}(n,k)$---introduced by Coxeter~\cite{Co50} and named by Watkins~\cite{Wa69}---are defined as follows. For $n \ge 5$ and $k < n/2$ put $${\rm GP}(n,k) = \left(\{u_i, v_i\}_{i=0}^{n-1}, \{ u_i u_{i+1}, u_i v_i, v_i v_{i+k} \}_{i=0}^{n-1} \right)\hspace{-1mm},$$ indices mod $n$. In this notation, the Petersen graph is ${\rm GP}(5,2)$.

Alspach~\cite{Al83} showed that a graph in ${\rm GP}(n,k)$ is non-hamiltonian if and only if $n \equiv 5 \pmod 6$ and $k = 2$. Bondy~\cite{Bo72} strengthened this result by proving that all members of this particular subfamily are in fact hypohamiltonian, and thus platypuses.

For $n \ge 5$ and $k < n/2$, we define the {\em Petersen prism} ${\rm PP}(n,k)$ as $${\rm PP}(n,k) = \left(\{u_i, v_i, w_i^1, w_i^2\}_{i=0}^{n-1}, \{ u_i u_{i+1}, u_i w_i^1, w_i^1 w_i^2, w_i^2 v_i, v_i v_{i+k} \}_{i=0}^{n-1} \right)\hspace{-1mm},$$ indices mod $n$.

\begin{figure}[h!tb]
	\centering
	\includegraphics[width=0.54\textwidth]{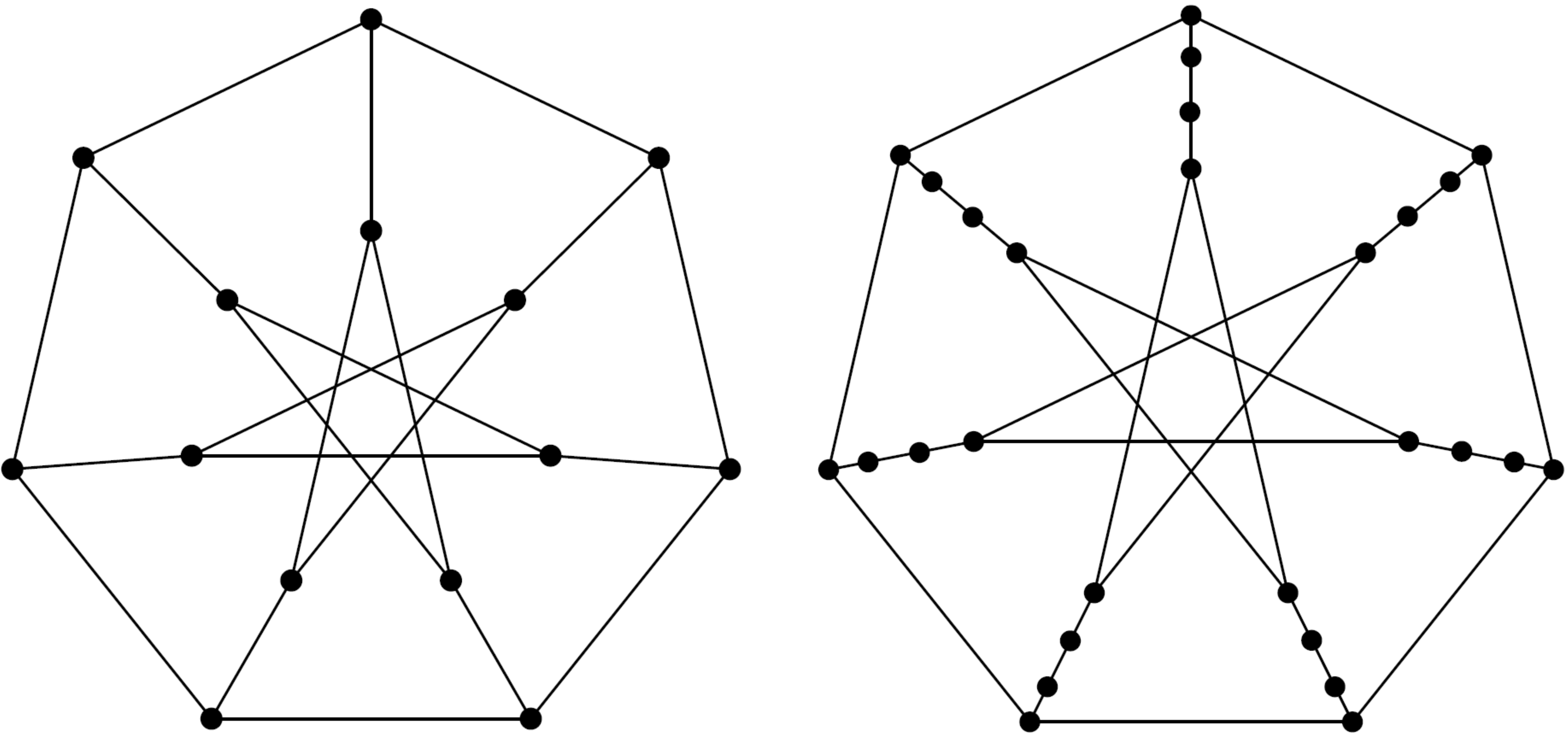}
	\caption{The generalised Petersen graph GP(7,3) and the Petersen prism PP(7,3).} 
	\label{fig:GP73_PP73}
\end{figure}

We call the edges of type $u_iv_i$ of a generalised Petersen graph \emph{spokes}. Then the associated Petersen prism can be obtained by adding two vertices on each spoke. Figure~\ref{fig:GP73_PP73} shows GP(7,3) and PP(7,3).

Let $G$ be a graph, consider its Cartesian product with $P_2$, i.e.\ $G \square P_2$, and replace each copy of $P_2$ with $P_3$. We will call the resulting graph the \emph{dotted prism over $G$} and denote it with $\dot{G}$. In a 2-connected graph $G$, we denote a path $P \subset G$ on $k \ge 3$ vertices and with end-vertices $v,w$ a \emph{$k$-ear} if $\{ v,w \}$ is a vertex-cut in $G$ and every vertex in $V(P) \setminus \{ v,w \}$ has degree~2 in $G$. We will also require an ear not to contain super-ears, i.e.\ for every ear $P$ there exists no ear $P'$ such that $P \subsetneq P'$.

It was shown in~\cite{Za17} that the dotted prism over a hamiltonian graph $G$ of odd order~$n \ge 3$ is a platypus, and that replacing any number of 3-ears with 4-ears in a dotted prism over an odd cycle yields a platypus. Let us call $D$ the operation of replacing in a given graph every 3-ear with a 4-ear. Then $D(\dot{C_n}) = {\rm PP}(n,1)$, where $C_n$ is the cycle of length $n$. So each member of ${\rm PP}(n,1)$ is a platypus for odd $n$. However, and this is easy to see, the maximum girth of a member of ${\rm PP}(n,1)$ is 8 (realised, for instance, by the graph ${\rm PP}(9,1) = D(\dot{C_9})$), one more than Coxeter's graph. Can we find platypuses of girth greater than 8? Yes, we can! But first we need two preparatory results:

\begin{lemma}\label{lemma:pp}
The Petersen prism ${\rm PP}(n,k)$ with $n$ odd is non-hamiltonian.
\end{lemma}

\begin{proof}
Assume it is hamiltonian. All vertices $w_i^1, w_i^2$ must be visited (i.e. each spoke must be traversed), which is only possible if the hamiltonian cycle contains every spoke, i.e.\ all paths $u_iw_i^1w_i^2v_i$. As $n$ is odd, there is an odd number of such paths, which yields a contradiction.
\end{proof}

\begin{theorem}\label{thm:pp}
The Petersen prism ${\rm PP}(n, 2)$ with $n$ odd is a platypus.
\end{theorem}

\begin{proof}
Consider $G \in \{ {\rm PP}(n, 2) : n \ge 5 \ \text{is odd} \}$. From Lemma~\ref{lemma:pp} it follows that $G$ is non-hamiltonian. We leave to the reader the easy proof that ${\rm PP}(5, 2)$ is a platypus and assume in the following that $n \ge 7$.
We use the notation from the definition of a Petersen prism, and all indices shall be taken mod $n$. For a path in $G$ containing a spoke $S = u_iw_i^1w_i^2v_i$, we write $S$ as $(uv)_i$ when traversed from $u_i$ to $v_i$, and as $(vu)_i$ if traversed from $v_i$ to $u_i$. Let
$$q = \begin{cases}
                            n - 1 & {\rm if} \ n = 1 \ {\rm mod} \ 4\\
                            n - 2 & {\rm if} \ n = 3 \ {\rm mod} \ 4.
                        \end{cases}$$
The path
$$(uv)_0(vu)_2(uv)_1(vu)_3u_4 ... (uv)_q$$
is a hamiltonian path in $G$ from $u_0$ to $v_q$. (To be clear: when $n = 3 \ {\rm mod} \ 4$, the paths ends with $... (uv)_{q-1}(vu)_{q+1}(uv)_q$.) By symmetry and the fact that $n$ is odd, we can conclude that $G - u_i$ and $G - v_i$ are traceable for all $i$.
A hamiltonian path in $G - w_0^1$ must have $w_0^2$ as an end-vertex. The path
$$w_0^2v_0(vu)_{n-2}(uv)_{n-3}(vu)_{n-1}u_0(uv)_1(vu)_3(uv)_2 ... (uv)_{q-3}$$
is hamiltonian in $G - w_0^1$. Similarly, a hamiltonian path in $G - w_0^2$ must have $w_0^1$ as an end-vertex. For $n = 1$ mod 4 consider
$$w_0^1u_0(uv)_1(vu)_{n-1}(uv)_{n-2}v_0(vu)_2(uv)_3(vu)_5(uv)_4(vu)_6 ... (vu)_{n-3},$$
while for $n = 3$ mod 4 consider
$$w_0^1u_0(uv)_{n-1}(vu)_{n-3}(uv)_{n-2}v_0(vu)_2(uv)_1(vu)_3 ... (vu)_{n-4}.$$
These are hamiltonian paths in $G - w_0^2$. Once more invoking the symmetry of the graph and the fact that $n$ is odd, it follows that $G - w_i^1$ and $G - w_i^2$ are traceable for all $i$. (The case $G = {\rm PP}(9,2)$ is illustrated in Figure~\ref{fig:platypus_pp92}.) This completes the proof.
\end{proof}

\begin{figure}[h!tb]
    \centering
 \includegraphics[width=0.9\textwidth]{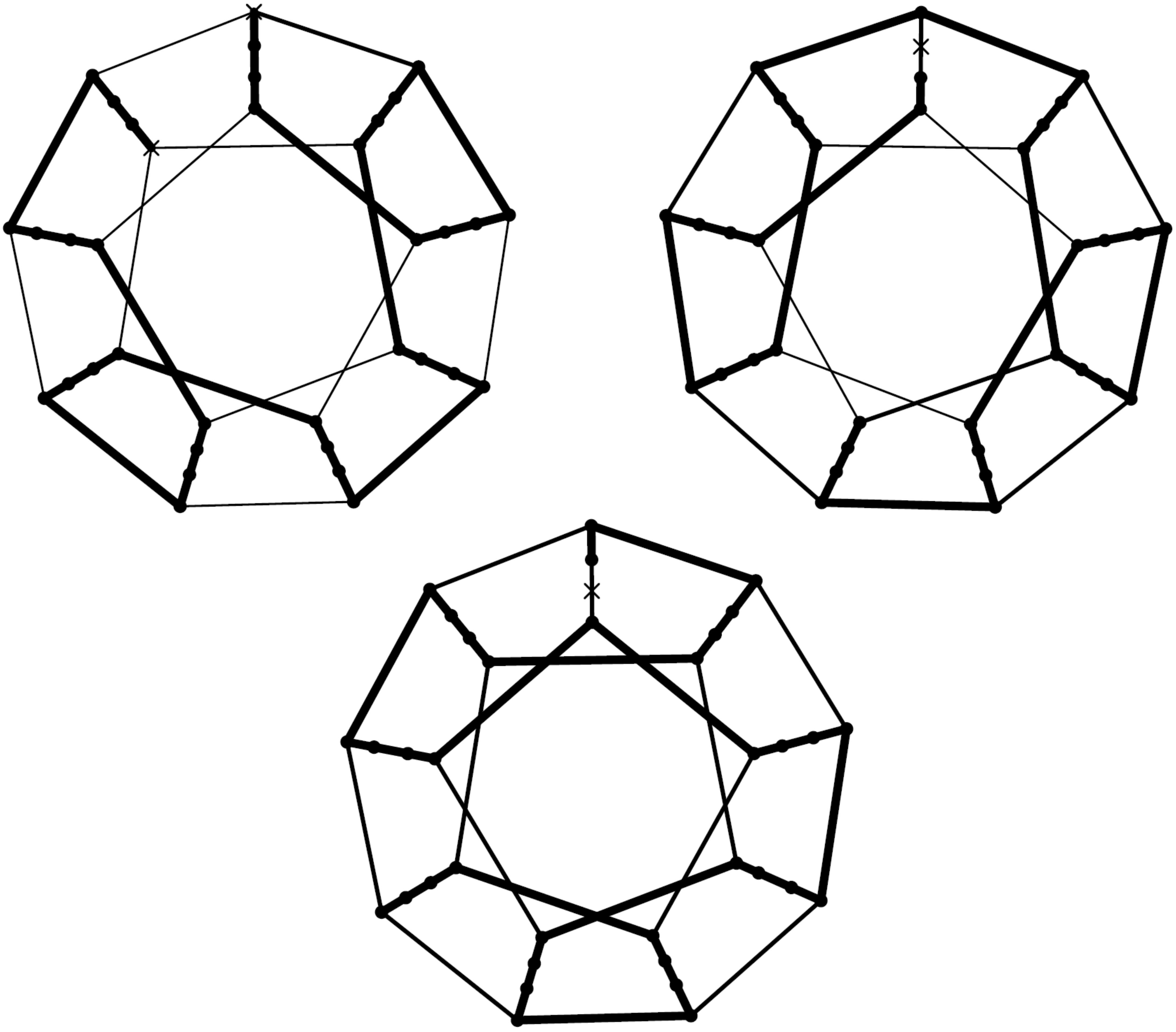}
    \caption{${\rm PP}(9,2)$ and the proof that all of its vertex-deleted subgraphs are traceable.}
    \label{fig:platypus_pp92}
\end{figure}

\begin{theorem}\label{thm:girth}
For every integer $g$ satisfying $3 \le g \le 16$ there exists a platypus of girth~$g$.
\end{theorem}

\begin{proof}
As discussed above, hypohamiltonian graphs (and thus platypuses) of all girths up to and including 7 are known~\cite{GZ17}. The Petersen prism ${\rm PP}(9,1)$ is a platypus of girth~8. For girth~9, consider $G_1 = {\rm PP}(9,2)$. We use Ferrero and Hanusch's \cite[Theorem 1.1]{FH14} (which is a practical restriction of a more general result due to Boben, Pisanski and \v{Z}irnik~\cite{BPZ05}) and obtain that $g({\rm GP}(9,2)) = 5$. Since 9 and 2 are relatively prime, any cycle of length at most~8 uses at least two spokes of ${\rm GP}(9,2)$, and \cite[Lemma 1.3]{FH14} of Ferrero and Hanusch yields that there exists a 5-cycle using exactly two spokes. Thus $g(G_1) = 9$. Theorem~\ref{thm:pp} implies that $G_1$ is a platypus.

For girth 10, consider $G_2 = {\rm PP}(11,3)$. Proceeding as above we obtain that $g({\rm GP}(11,3)) = 6$. Since 11 and 3 are relatively prime, any cycle of length at most~10 uses at least two spokes of $G_2$. By \cite[Lemma 1.3]{FH14}, there exists a 6-cycle using exactly two spokes. Thus $g(G_2) = 10$. By Lemma~\ref{lemma:pp}, $G_2$ is non-hamiltonian. The reasoning that every vertex-deleted subgraph of $G_2$ is traceable is very similar to what was presented in the proof of Theorem~\ref{thm:pp} and Figure~\ref{fig:platypus_pp92} and therefore omitted. We also skip the analysis for girths 11 and 12, and only mention that it suffices to consider ${\rm PP}(13,5)$ and ${\rm PP}(23,5)$, respectively---the arguments are similar to the ones given above.

For girth 13, consider $G_3 = {\rm PP}(31,7)$. By \cite[Theorem 1.1]{FH14}, we have that $g({\rm GP}(31,7)) = 8$. A cycle in ${\rm GP}(31,7)$ using no spokes has length at least 31, and it is straightforward to verify that a cycle in ${\rm GP}(31,7)$ containing exactly two spokes has length at least 9, and such a 9-cycle exists, take e.g.\ $$v_1v_8v_{15}v_{22}v_{29}u_{29}u_{30}u_{31}u_1.$$ Clearly, any cycle in $G_3$ using four or more spokes has length at least 16. Thus $g(G_3) = 13$. $G_3$ is non-hamiltonian by Lemma~\ref{lemma:pp}. We leave to the reader the routine verification that every vertex-deleted subgraph of $G_3$ is traceable, which is akin to the one given in Figure~\ref{fig:platypus_pp92}.

Similar arguments yield that ${\rm PP}(39,7)$, ${\rm PP}(49,9)$, and ${\rm PP}(59,9)$ are platypuses of girth 14, 15, and 16, respectively. We also wrote a computer program which constructs the Petersen prism ${\rm PP}(n,k)$ for given $n$ and $k$ and also verified by computer that the Petersen prisms mentioned in this proof indeed have the desired girth.
\end{proof}

In above theorem we have not discussed the problem of finding the \emph{smallest} (regarding order) platypus of a particular girth---a computational approach will allow us to do so in Section~\ref{sect:computations}.

By Ferrero and Hanusch's \cite[Theorem 1.1 and Lemma 1.3]{FH14}, we can conclude that the largest girth obtainable through the approach using Petersen prisms is 16: as girth 8 is the maximum possible girth in a generalised Petersen graph, there are always cycles of length 8 using exactly four spokes, e.g.\ in ${\rm GP}(n,k)$ the cycle $$u_1v_1v_{k+1}u_{k+1}u_{k+2}v_{k+2}v_2u_2,$$ and by \cite[Proposition 2.1(v)]{Za17} we cannot add more vertices on a spoke and remain in the family of platypuses. For planar graphs, the situation is dramatically different:

\begin{theorem}
A planar platypus has girth at most~$9$.
\end{theorem}

\begin{proof}
We need the following simple yet crucial fact first proven in~\cite[Proposition~2.1(v)]{Za17}: in a platypus, any vertex has at most one neighbour of degree~2. Assume there exists a graph $G$ obtained from a planar platypus of girth at least~10 by suppressing all vertices of degree~2 (to be clear: a path $uvv'w$ in the planar platypus, where the degrees of $v$ and $v'$ are 2 and the degrees of $u$ and $w$ are at least 3, becomes $uw$; here, $v$ and $v'$ may coincide). $G$ is a planar graph with minimum degree at least~3 and girth at least~6, by the aforementioned fact. But this contradicts Euler's formula.
\end{proof}

Note that a planar platypus of girth $g$ exists for every $g \in \{ 3, ..., 8 \}$: $\dot{C}_3$, $D(\dot{C}_5) = {\rm PP}(5,1)$ plus an edge between the two end-vertices of a 4-ear, $\dot{C}_5$, $\dot{C}_7$, $D(\dot{C}_7) = {\rm PP}(7,1)$, and $D(\dot{C}_9) = {\rm PP}(9,1)$, respectively. However, it remains unknown whether a planar platypus of girth 9 exists.

\section{Maximum degree of platypuses}\label{sect:maxdeg}

A graph $G$ is called \emph{homogeneously traceable} if every vertex of the graph is an end-vertex of a hamiltonian path. It is easy to see that every non-hamiltonian homogeneously traceable graph is a platypus, but not every platypus is a non-hamiltonian homogeneously graph---simply consider non-traceable platypuses, i.e.\ hypotraceable graphs. Chartrand, Gould, and Kapoor proved in~\cite[Theorem~1]{CGK79} that if $G$ is a non-hamiltonian homogeneously traceable graph, then $\Delta(G) \le n - 4$. The first part of the following theorem is a generalisation of this result.

\begin{theorem} \label{thm:maxdeg}
If $G$ is a platypus of order $n$, then $\Delta(G) \le n - 4$. For $n \le 9$ there are precisely four platypuses, each of order~$9$, one of maximum degree~$3$ and the other three of maximum degree~$4$, hence, the bound is not sharp for $n = 9$. However, for all $n \ge 10$ there exists a platypus $G$ of order $n$ with $\Delta(G) = n - 4$.
\end{theorem}

\begin{proof}
We now prove the first part of the statement. We skip the simple case $\Delta(G) = n - 1$. Assume $G$ has maximum degree $n - 2$, and let $v \in V(G)$ be a vertex of this degree. Let ${\mathfrak p}_v$ be a hamiltonian path in $G - v$ with end-vertices $u$ and $w$. Since $\deg(v) = n - 2$, there exists exactly one vertex $x \in V({\mathfrak p}_v)$ such that $v$ and $x$ are non-adjacent. If $x \notin \{ u,w \}$, we contradict the non-hamiltonicity of $G$ (consider ${\mathfrak p}_v \cup uvw$), so consider w.l.o.g.\ $x = u$. 
As $G$ is 2-connected, $x$ has a neighbour $y$ in $V({\mathfrak p}_v)$ which is not the (unique) neighbour of $x$ on ${\mathfrak p}_v$.
Now consider the subpath of ${\mathfrak p}_v$ between $x$ and $y$ and remove $y$ to obtain a path ${\mathfrak p}'_v$ with end-vertices $x$ and $y'$. Let ${\mathfrak p}''_v$ be the subpath of ${\mathfrak p}_v$ between $y$ and $w$. Then ${\mathfrak p}'_v \cup {\mathfrak p}''_v \cup y'vw + yx$ is a hamiltonian cycle in $G$, a contradiction.

Now suppose $G$ has maximum degree $n - 3$, and let $v \in V(G)$ be a vertex of degree $n - 3$. There exist two vertices $x, x'$ in $G - v$ that are non-adjacent to $v$. As above, let ${\mathfrak p}_v$ be a hamiltonian path in $G - v$ with end-vertices $u$ and $w$. If neither $x$ nor $x'$ is an end-vertex of ${\mathfrak p}_v$, we are done, so assume that $x = u$ is an end-vertex of ${\mathfrak p}_v$ but $x'$ is not. This can be dealt with as in the first paragraph of this proof, unless the second neighbour $y$ of $x$ is adjacent to $x'$ and lies on the subpath of ${\mathfrak p}_v$ between $x'$ and $w$, and $x$ has no further neighbours. 
Consider the hamiltonian path ${\mathfrak p}_y$ in $G - y$. ${\mathfrak p}_y$ has $x$ as an end-vertex, but cannot have $x'$ as an end-vertex, since then ${\mathfrak p}_y \cup xyx'$ would be a hamiltonian cycle in $G$. Thus, $x'$ has a neighbour $z$ which is not adjacent to $x'$ on ${\mathfrak p}_v$. If $z$ lies on the subpath of ${\mathfrak p}_v$ between $w$ and $y$, then, writing ${\mathfrak p}_{ab}$ for the subpath of ${\mathfrak p}_v$ between vertices $a$ and $b$ on ${\mathfrak p}_v$, we obtain a contradiction through the hamiltonian cycle $${\mathfrak p}_{xx'} +x'z \cup {\mathfrak p}_{zw} \cup wvz' \cup {\mathfrak p}_{z'y} + yx,$$ where $z'$ is the neighbour of $y$ on ${\mathfrak p}_v$ which is not $x'$. If $z$ lies on the subpath of ${\mathfrak p}_v$ between $x'$ and $x$, a very similar argument yields a contradiction, as well---recall that $z$ is not adjacent to $x'$ on ${\mathfrak p}_v$.

So let ${\mathfrak p}_v = v_1 ... v_{n-1}$ with $x = v_1$ and $x' = v_{n-1}$. Since $G$ is 2-connected, $x$ and $x'$ have neighbours $v_i$ and $v_j$ in $G - v$, respectively, where $i \ne 2$ and $j \ne n - 2$. We now present the five essentially different cases, all of which lead to a hamiltonian cycle and thus a contradiction. If $i \le j$, we have $v v_{i-1} ... v_1 v_i ... v_j v_{n-1} ... v_{j+1} v$. If $i > j$, $i \ne n - 2$ and $j \ne 2$, then $v v_{j-1} ... v_1 v_i ... v_j v_{n-1} ... v_{i+1} v$ is a hamiltonian cycle in $G$, again a contradiction. If $i > j$ and $j = 2$, then $v v_3 ... v_i v_1 v_2 v_{n-1} ... v_{i+1} v$ leads to a contradiction. (The case $i > j$ and $i = n-2$ is analogous.) Consider the case $i > j$ with $i = n-2$ and $j = 2$. 
Consider the path ${\mathfrak p}$ in $G - v_2$. If the end-vertices of ${\mathfrak p}$ are $v_1$ and $v_{n-1}$, we are done due to the cycle ${\mathfrak p} \cup v_1v_2v_{n-1}$. Therefore, $v_{n-1}$ must have a neighbour $v_k$ other than $v_2$ and $v_{n-2}$. But then $v_1...v_{k-1}vv_{k+1}v_kv_{n-1}v_{n-2}v_1$ leads to a contradiction.
Finally, consider the case that $v_1$ and $v_{n-1}$ are adjacent. In this case it is trivial to obtain a contradiction.

That the smallest platypuses have order 9, that there exist precisely four platypuses of that order, and that they satisfy the properties given in the statement was proven in~\cite{Za17} by Van Cleemput and the third author. It also follows by inspecting the complete list of platypus graphs up to 12~vertices presented in Section~\ref{sect:computations}.

The final part of the theorem is a direct consequence of a construction of Chartrand, Gould, and Kapoor: they prove in~\cite{CGK79} that for every $n \ge 10$ there exists a non-hamiltonian homogeneously traceable graph (and thus a platypus) of order $n$ and with maximum degree $n-4$.
\end{proof}

Combining Lemma~\ref{lemma:maxdeg} with Theorem~\ref{thm:maxdeg} we obtain the following.

\begin{corollary}
A maximally non-hamiltonian graph of order $n$ cannot have maximum degree $n-2$ or $n-3$.
\end{corollary}


We end this section by commenting on another theorem of Chartrand, Gould, and Kapoor~\cite[Theorem~2]{CGK79}: they showed that if $G$ is a homogeneously traceable non-hamiltonian graph, then every two vertices of degree $n - 4$ are adjacent. Using an approach similar to theirs, we can extend the result and show that this holds for platypuses, as well.

\begin{theorem}
In a platypus of order $n$, any two vertices of degree $n - 4$ are adjacent.
\end{theorem}

\begin{proof}
Assume there exists a platypus $G$ of order $n$ with two non-adjacent vertices $v$ and $w$, each of degree $n - 4$. We first show that there exists a hamiltonian $vw$-path in $G$. Assume there is no such path. Then $G + vw$ is a platypus of order $n$ and maximum degree $n-3$, in contradiction with Theorem~\ref{thm:maxdeg}.

So $G$ contains a hamiltonian $vw$-path which we write as $v v_2 v_3 ...v_{n-2} v_{n-1} w$. Like Chartrand, Gould, and Kapoor, we now use the fact that $G$ must contain edges $vv_\ell$ and $wv_{\ell-1}$ for an appropriate $\ell \in \{ 3, ..., n-1 \}$. We call, ad hoc, such edges \emph{good} and prove the aforementioned fact by assuming that $G$ contains no good edges.

Consider $v_i \in N(w)$ with $i$ minimal. Since $G$ contains no good edges, $v_{i+1} \notin N(v)$. Then there exists a minimal $j > i$ such that $v_j \in N(w)$. Again, $v_{j+1} \notin N(v)$. Finally, there must exist a minimal $k > j$ such that $v_k \in N(w)$ (since $|N(w)| = n - 4 \ge 5$). Once more, $v_{k+1} \notin N(v)$. Since $v v_2 v_3 ...v_{n-2} v_{n-1} w$ is a path and $i < j < k$, the vertices $v_{i+1}, v_{j+1}, v_{k+1}$ must be pairwise distinct. However, this is impossible: we have $\{ v_{i+1}, v_{j+1}, v_{k+1} \} \subset V(G) \setminus ( \{ v \} \cup \{ w \} \cup N(v) )$, but the latter set has only $n - (n-2) = 2$ elements. We have proven that there always exist good edges $vv_\ell$ and $wv_{\ell-1}$. But then $v v_2 ... v_{\ell - 1} w v_{n-1} ... v_{\ell} v$ is a hamiltonian cycle in $G$, a contradiction.
\end{proof}

\section{Computational results}
\label{sect:computations}

It is straightforward to modify our generation algorithm for hypohamiltonian graphs from~\cite{GZ17} to generate all pairwise non-isomorphic platypuses of a given order. The details of this modified generation algorithm for platypuses can be found in the master's thesis of the second author~\cite{Ne17}.

Our implementation of this algorithm is incorporated in our generator for hypohamiltonian graphs and can be downloaded from~\cite{genhypo-site}. Using this, we were able to generate complete lists of platypuses for various orders and lower bounds on the girth. The counts of these graphs are given in Table~\ref{table:counts_platypuses}. All graphs from this table can also be downloaded from the \textit{House of Graphs}~\cite{hog} at \url{http://hog.grinvin.org/Platypus}

		\begin{table}[!htb]
			\centering
			\begin{tabular}{cccccccc}
				\toprule
				Order & \# platypuses & $g \geq 4$ & $ g \geq 5$ & $ g \geq 6$ & $ g \geq 7$ & $ g \geq 8$ & $ g \geq 9$\\
				\midrule
				$0-8$ & $0$        & $0$    & $0$     & $0$    & $0$  & $0$       & $0$  \\
				$9$   & $4$        & $0$    & $0$     & $0$    & $0$  & $0$       & $0$  \\
				$10$  & $48$       & $2$    & $2$     & $0$    & $0$  & $0$       & $0$  \\
				$11$  & $814$      & $4$    & $3$     & $0$    & $0$  & $0$       & $0$  \\
				$12$  & $24847$    & $48$   & $7$     & $1$    & $0$  & $0$       & $0$  \\
				$13$  & ?          & $319$  & $27$    & $1$    & $0$  & $0$       & $0$  \\
				$14$  & ?          & $6623$ & $161$   & $2$    & $0$  & $0$       & $0$  \\
				$15$  & ?          & ?      & $934$   & $1$    & $0$  & $0$       & $0$  \\
				$16$  & ?          & ?      & $7674$  & $9$    & $1$  & $0$       & $0$  \\
				$17$  & ?          & ?      & $82240$ & $53$   & $0$  & $0$       & $0$  \\
				$18$  & ?          & ?      & ?       & $277$  & $0$  & $0$       & $0$  \\
				$19$  & ?          & ?      & ?       & $1161$ & $0$  & $0$       & $0$  \\
				$20$  & ?          & ?      & ?       & $7659$ & $5$  & $0$       & $0$  \\
				$21$  & ?          & ?      & ?       & ?      & $35$ & $0$       & $0$  \\
				$22$  & ?          & ?      & ?       & ?      & ?    & $1$       & $0$  \\
				$23$  & ?          & ?      & ?       & ?      & ?    & $1$       & $0$  \\
				$24$  & ?          & ?      & ?       & ?      & ?    & $5$       & $0$  \\
				\bottomrule
				\end{tabular}
				\caption{The number of platypuses. The columns with a header of the form $g \geq k$ contain the number of platypuses with girth at least~$k$. Note that we do know of platypuses of girths 9, ..., 16, see Theorem~\ref{thm:girth}}
				\label{table:counts_platypuses}
				\end{table}

The algorithm from~\cite{Ne17} was implemented in two independent ways and all results we obtained with it were confirmed by both implementations. Next to that, as an additional correctness test we independently verified all results for the smaller orders by using the generator \verb|geng|~\cite{nauty-website, mckay_14} to generate all graphs and then filtering the platypuses.

The four smallest platypuses, each of girth~3, were given in~\cite{Za17}. The smallest platypus of girth~4 is shown in Figure~\ref{fig:platypus_g4} (and has order~11). The smallest platypuses of girth~5 are Petersen's graph $P$ and $P$ minus an edge. A platypus of girth 6 was published in Figure~1b from~\citep{Za17} and through Table~\ref{table:counts_platypuses} we here prove that this platypus is the smallest one of girth~6. The smallest platypuses of girth 7 and 8 are shown in Figure~\ref{fig:platypus_g7} and~\ref{fig:platypus_g8}, respectively. Note that the former is a so-called modified dotted prism, defined in~\cite{Za17}, in which every 3-ear was replaced with a 4-ear.

\begin{figure}[h!tb]
	\centering
	\includegraphics[width=0.25\textwidth]{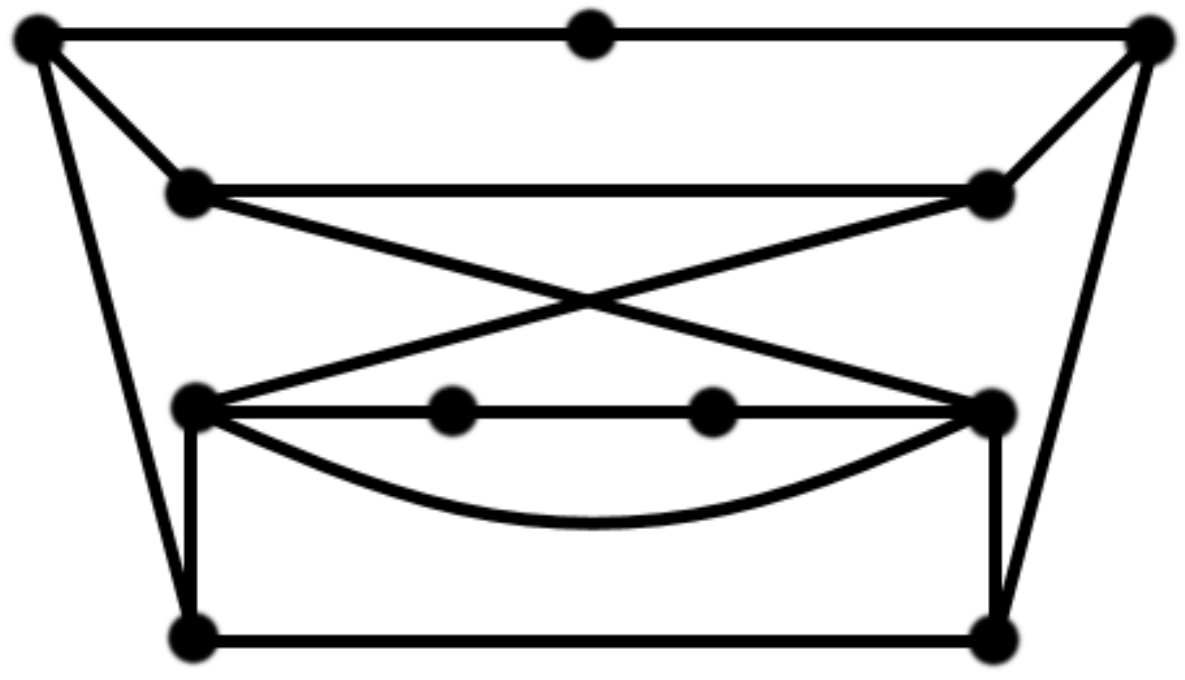}
    \caption{The smallest platypus of girth 4. It has order 11.}
    \label{fig:platypus_g4}
\end{figure}

\begin{figure}[h!tb]
    \centering
   \subfloat[]{\label{fig:platypus_g7}\includegraphics[width=0.22\textwidth, angle=90]{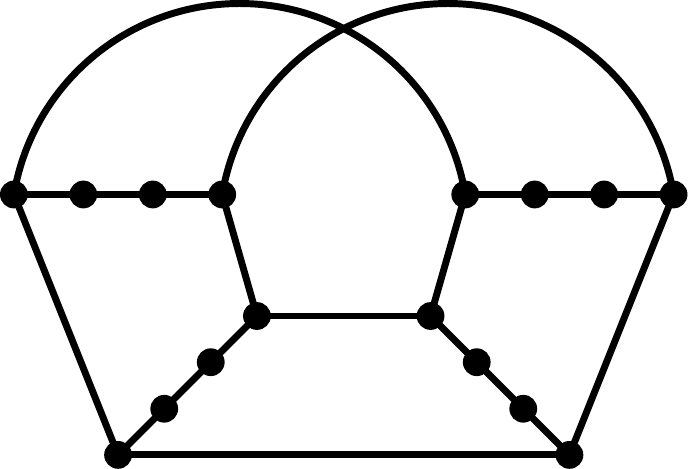}} \qquad
    \subfloat[]{\label{fig:platypus_g8}\includegraphics[width=0.3\textwidth]{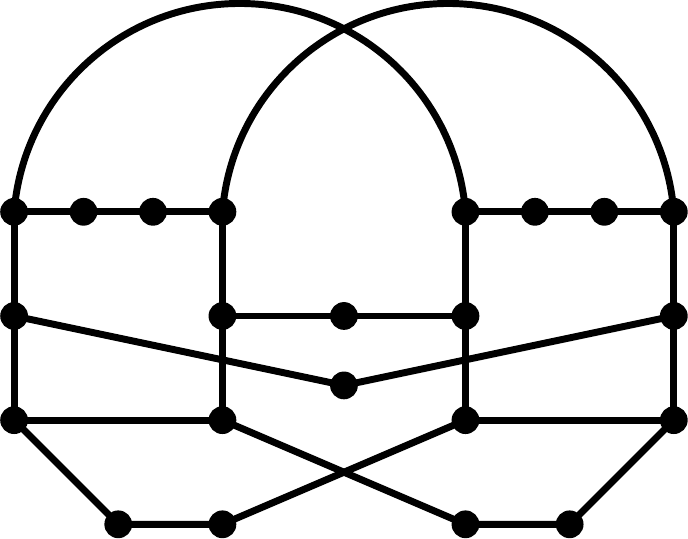}}
    \caption{(a) The smallest platypus of girth 7 and (b) girth 8. Their order is 16 and 22, respectively.}
\end{figure}

\section{On two problems of Zamfirescu}
\label{sect:problems}


The third author observed~\cite{Za17} that if a platypus contains a 4-ear, then this 4-ear can be replaced with a 3-ear, and the resulting graph is still a platypus. In the same paper, he raised the question whether this holds in the inverse direction, as well. We now show that it does not.

\begin{proposition}
There exists a platypus containing a $3$-ear, whose replacement with a $4$-ear does not yield a platypus.
\end{proposition}

\begin{proof}
See Figure~\ref{fig:ear}. We leave to the reader the straightforward verification that the graph shown on the left-hand side of Figure~\ref{fig:ear} is a platypus, but the graph shown on the right-hand side is not.
\end{proof}

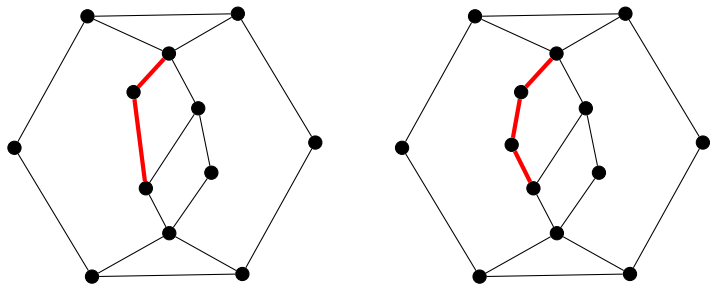
\begin{figure}[h!tb]
	\centering
\begin{tikzpicture}[every node/.style={circle, draw, inner sep=1pt,
minimum size=5pt},every edge/.style={draw},scale=0.04]
    \definecolor{marked}{rgb}{0.25,0.5,0.25}
    \node [circle,fill] (13) at (39.587416,67.641954) {};
    \node [circle,fill] (12) at (51.467468,20.702965) {};
    \node [circle,fill] (11) at (51.345873,80.474863) {};
    \node [circle,fill] (10) at (43.671669,35.637131) {};
    \node [circle,fill] (9) at (25.772817,6.267250) {};
    \node [circle,fill] (8) at (99.999999,50.861034) {};
    \node [circle,fill] (7) at (24.261332,92.860197) {};
    \node [circle,fill] (6) at (65.448953,40.833281) {};
    \node [circle,fill] (4) at (75.761300,7.139802) {};
    \node [circle,fill] (3) at (0.000000,49.115528) {};
    \node [circle,fill] (2) at (74.249814,93.732749) {};
    \node [circle,fill] (1) at (61.088539,62.256461) {};    
    \draw [ultra thick, red] (13) to (11);
    \draw [black] (12) to (4);
    \draw [black] (12) to (9);
    \draw [black] (12) to (10);
    \draw [black] (12) to (6);
    \draw [black] (11) to (1);
    \draw [black] (11) to (7);
    \draw [black] (11) to (2);
    \draw [black] (10) to (1);
    \draw [ultra thick, red] (10) to (13);
    \draw [black] (9) to (4);
    \draw [black] (9) to (3);
    \draw [black] (8) to (2);
    \draw [black] (8) to (4);
    \draw [black] (7) to (3);
    \draw [black] (7) to (2);
    \draw [black] (6) to (1);
\end{tikzpicture}
\qquad
\begin{tikzpicture}[every node/.style={circle, draw, inner sep=1pt,
minimum size=5pt},every edge/.style={draw},scale=0.04]
    \definecolor{marked}{rgb}{0.25,0.5,0.25}
    \node [circle,fill] (13) at (39.587416,67.641954) {};
    \node [circle,fill] (12) at (51.467468,20.702965) {};
    \node [circle,fill] (11) at (51.345873,80.474863) {};
    \node [circle,fill] (10) at (43.671669,35.637131) {};
    \node [circle,fill] (9) at (25.772817,6.267250) {};
    \node [circle,fill] (8) at (99.999999,50.861034) {};
    \node [circle,fill] (7) at (24.261332,92.860197) {};
    \node [circle,fill] (6) at (65.448953,40.833281) {};
    \node [circle,fill] (5) at (36.437039,50.115528) {};
    \node [circle,fill] (4) at (75.761300,7.139802) {};
    \node [circle,fill] (3) at (0.000000,49.115528) {};
    \node [circle,fill] (2) at (74.249814,93.732749) {};
    \node [circle,fill] (1) at (61.088539,62.256461) {};
    \draw [ultra thick, red] (13) to (5);
    \draw [ultra thick, red] (13) to (11);
    \draw [black] (12) to (4);
    \draw [black] (12) to (9);
    \draw [black] (12) to (10);
    \draw [black] (12) to (6);
    \draw [black] (11) to (1);
    \draw [black] (11) to (7);
    \draw [black] (11) to (2);
    \draw [black] (10) to (1);
    \draw [ultra thick, red] (10) to (5);
    \draw [black] (9) to (4);
    \draw [black] (9) to (3);
    \draw [black] (8) to (2);
    \draw [black] (8) to (4);
    \draw [black] (7) to (3);
    \draw [black] (7) to (2);
    \draw [black] (6) to (1);
\end{tikzpicture}

 	\caption{There exist platypuses containing a 3-ear (left-hand side), the replacement of which with a 4-ear yields a graph which is not a platypus (right-hand side). The 3-ear which is replaced by a 4-ear is marked in bold red.}
	\label{fig:ear}
\end{figure}

The third author defined in~\cite{Za17} the number $\psi_k$ $\left(\bar{\psi}_k\right)$ as the order of the smallest platypus (smallest planar platypus) of connectivity~$k$, where, naturally, the polyhedral case $\bar{\psi}_3$ is of special interest. In~\cite{Za17} it was proven that $\bar{\psi}_3 \le 25$. In the following, we provide the first non-trivial lower bound and improve the upper bound for $\bar{\psi}_3$.

Using the program \verb|plantri|~\cite{BM05,BM07} we generated all planar 3-connected graphs with a given lower bound on the girth and tested (using two independent programs) which of the resulting graphs are platypuses. This yielded the following results.

\begin{theorem} \label{claim:platypus_polyhedral}
The following holds:
\begin{enumerate}[(i)]
\item The smallest planar $3$-connected platypus graph has at least $18$ vertices.
\item The smallest planar $3$-connected platypus graph of girth $4$ has $21$ vertices and is shown in Figure~\ref{fig:planar_3conn_g4}.
\item The smallest planar $3$-connected platypus graph of girth $5$ has $28$ vertices and is shown in Figure~\ref{fig:planar_3conn_g5}.
\end{enumerate}
\end{theorem}

This implies the following bounds for $\bar{\psi}_3$.

\begin{theorem} \label{thm:polyhedral}
We have that $18 \le \bar{\psi}_3 \le 21$.
\end{theorem}

The smallest planar 3-connected platypus of girth 3 known so far has 25 vertices and is depicted in Figure~4 of~\cite{Za17}. This gives us the following bounds for the smallest planar 3-connected platypus of girth 3.

\begin{corollary}
The smallest planar $3$-connected platypus graph with girth $3$ has at least $18$ and at most $25$ vertices.
\end{corollary}

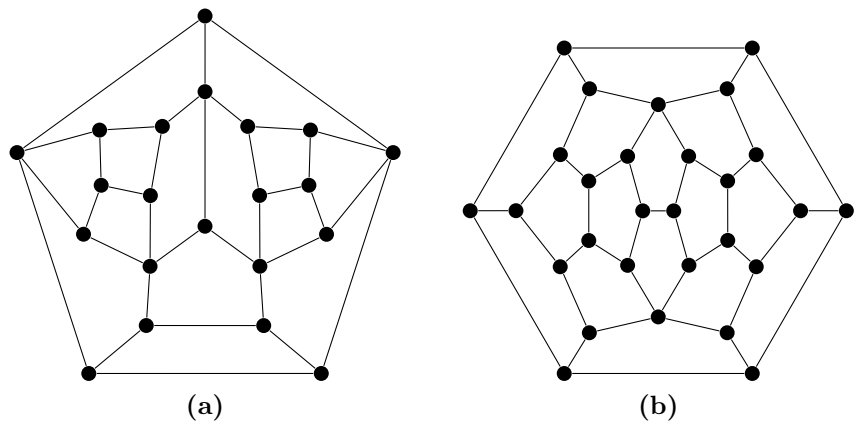
\begin{figure}[h!tb]
	\centering
\subfloat[]{\label{fig:planar_3conn_g4}\begin{tikzpicture}[scale=0.05, inner sep=2pt]
    \definecolor{marked}{rgb}{0.25,0.5,0.25}
    \node [circle,fill] (21) at (22.378600,52.523424) {};
    \node [circle,fill] (20) at (35.461542,49.788391) {};
    \node [circle,fill] (19) at (17.680894,39.494396) {};
    \node [circle,fill] (18) at (21.955510,67.203605) {};
    \node [circle,fill] (17) at (38.643556,68.192232) {};
    \node [circle,fill] (16) at (35.368859,30.961922) {};
    \node [circle,fill] (15) at (34.359609,15.208089) {};
    \node [circle,fill] (14) at (19.080609,2.444254) {};
    \node [circle,fill] (13) at (0.000000,61.224905) {};
    \node [circle,fill] (12) at (49.986908,77.427577) {};
    \node [circle,fill] (11) at (49.977255,41.626327) {};
    \node [circle,fill] (10) at (65.589381,15.190900) {};
    \node [circle,fill] (9) at (80.895282,2.451042) {};
    \node [circle,fill] (8) at (49.990670,97.555745) {};
    \node [circle,fill] (7) at (61.337746,68.191034) {};
    \node [circle,fill] (6) at (64.574428,30.975874) {};
    \node [circle,fill] (5) at (99.999999,61.221841) {};
    \node [circle,fill] (4) at (78.022712,67.225041) {};
    \node [circle,fill] (3) at (64.495859,49.814303) {};
    \node [circle,fill] (2) at (82.287701,39.474423) {};
    \node [circle,fill] (1) at (77.601616,52.528401) {};
    \draw [black] (21) to (18);
    \draw [black] (21) to (20);
    \draw [black] (21) to (19);
    \draw [black] (20) to (16);
    \draw [black] (20) to (17);
    \draw [black] (19) to (13);
    \draw [black] (19) to (16);
    \draw [black] (18) to (13);
    \draw [black] (18) to (17);
    \draw [black] (17) to (12);
    \draw [black] (16) to (15);
    \draw [black] (16) to (11);
    \draw [black] (15) to (10);
    \draw [black] (15) to (14);
    \draw [black] (14) to (13);
    \draw [black] (14) to (9);
    \draw [black] (13) to (8);
    \draw [black] (12) to (7);
    \draw [black] (12) to (11);
    \draw [black] (12) to (8);
    \draw [black] (11) to (6);
    \draw [black] (10) to (9);
    \draw [black] (10) to (6);
    \draw [black] (9) to (5);
    \draw [black] (8) to (5);
    \draw [black] (7) to (3);
    \draw [black] (7) to (4);
    \draw [black] (6) to (2);
    \draw [black] (6) to (3);
    \draw [black] (5) to (4);
    \draw [black] (5) to (2);
    \draw [black] (4) to (1);
    \draw [black] (3) to (1);
    \draw [black] (2) to (1);
\end{tikzpicture}}
\qquad
\subfloat[]{\label{fig:planar_3conn_g5}\begin{tikzpicture}[scale=0.05, inner sep=2pt]
    \node [circle,fill] (28) at (31.519726,42.080900) {};
    \node [circle,fill] (27) at (31.519726,57.904857) {};
    \node [circle,fill] (26) at (45.862413,49.971514) {};
    \node [circle,fill] (25) at (54.109101,49.971514) {};
    \node [circle,fill] (24) at (41.845891,35.472155) {};
    \node [circle,fill] (23) at (23.928215,35.101837) {};
    \node [circle,fill] (22) at (12.163511,50.014243) {};
    \node [circle,fill] (21) at (23.970944,64.926648) {};
    \node [circle,fill] (20) at (41.845891,64.485115) {};
    \node [circle,fill] (19) at (58.068651,64.513602) {};
    \node [circle,fill] (18) at (58.154109,35.500640) {};
    \node [circle,fill] (17) at (49.992879,21.756159) {};
    \node [circle,fill] (16) at (31.676399,17.540235) {};
    \node [circle,fill] (15) at (0.000000,50.000000) {};
    \node [circle,fill] (14) at (31.719128,82.459764) {};
    \node [circle,fill] (13) at (50.007122,78.243840) {};
    \node [circle,fill] (12) at (68.437544,57.933342) {};
    \node [circle,fill] (11) at (68.466031,42.109385) {};
    \node [circle,fill] (10) at (68.309358,17.540235) {};
    \node [circle,fill] (9) at (24.996439,6.701325) {};
    \node [circle,fill] (8) at (24.996439,93.298675) {};
    \node [circle,fill] (7) at (68.309358,82.459764) {};
    \node [circle,fill] (6) at (76.029055,64.912406) {};
    \node [circle,fill] (5) at (76.043299,35.101837) {};
    \node [circle,fill] (4) at (74.989318,6.701325) {};
    \node [circle,fill] (3) at (74.989318,93.298675) {};
    \node [circle,fill] (2) at (87.822246,50.000000) {};
    \node [circle,fill] (1) at (99.999999,50.000000) {};
    \draw [black] (28) to (23);
    \draw [black] (28) to (27);
    \draw [black] (28) to (24);
    \draw [black] (27) to (21);
    \draw [black] (27) to (20);
    \draw [black] (26) to (24);
    \draw [black] (26) to (20);
    \draw [black] (26) to (25);
    \draw [black] (25) to (18);
    \draw [black] (25) to (19);
    \draw [black] (24) to (17);
    \draw [black] (23) to (22);
    \draw [black] (23) to (16);
    \draw [black] (22) to (15);
    \draw [black] (22) to (21);
    \draw [black] (21) to (14);
    \draw [black] (20) to (13);
    \draw [black] (19) to (12);
    \draw [black] (19) to (13);
    \draw [black] (18) to (17);
    \draw [black] (18) to (11);
    \draw [black] (17) to (10);
    \draw [black] (17) to (16);
    \draw [black] (16) to (9);
    \draw [black] (15) to (8);
    \draw [black] (15) to (9);
    \draw [black] (14) to (8);
    \draw [black] (14) to (13);
    \draw [black] (13) to (7);
    \draw [black] (12) to (11);
    \draw [black] (12) to (6);
    \draw [black] (11) to (5);
    \draw [black] (10) to (4);
    \draw [black] (10) to (5);
    \draw [black] (9) to (4);
    \draw [black] (8) to (3);
    \draw [black] (7) to (6);
    \draw [black] (7) to (3);
    \draw [black] (6) to (2);
    \draw [black] (5) to (2);
    \draw [black] (4) to (1);
    \draw [black] (3) to (1);
    \draw [black] (2) to (1);
\end{tikzpicture}}

    \caption{The smallest planar 3-connected platypus of girth 4 and 5. They have 21 and 28 vertices, respectively.}

\end{figure}

It was proven in~\cite{Za17} that there exists a polyhedral platypus of order $n$ for every $n \ge 25$. We now improve this theorem.

\begin{theorem}
There exists a polyhedral platypus of order $n$ for every $n \ge 21$.
\end{theorem}

\begin{proof}
It follows from Theorem~\ref{claim:platypus_polyhedral}~(ii) that there is a polyhedral platypus of order 21. Figures~\ref{fig:planar_3conn_g4_22v} and~\ref{fig:planar_3conn_g4_23v} show a polyhedral platypus of order 22 and 23 vertices, respectively. We leave the proof that these graphs are indeed platypuses to the reader.  Figure~3.8 of~\cite{Ne17} shows a polyhedral platypus on 24 vertices. So together with the fact that there exists a polyhedral platypus of order $n$ for every $n \ge 25$ as shown in~\cite{Za17}, this proves the theorem.
\end{proof}

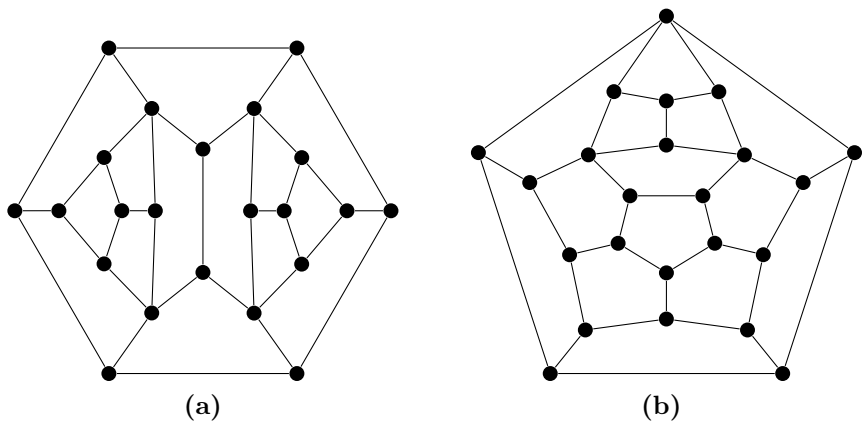
\begin{figure}[h!tb]
	\centering       	
\subfloat[]{\label{fig:planar_3conn_g4_22v}\begin{tikzpicture}[scale=0.05, inner sep=2pt]
    \definecolor{marked}{rgb}{0.25,0.5,0.25}
    \node [circle,fill] (22) at (23.738669,64.178210) {};
    \node [circle,fill] (21) at (28.441935,50.017102) {};
    \node [circle,fill] (20) at (11.715408,50.017102) {};
    \node [circle,fill] (19) at (36.411834,77.227637) {};
    \node [circle,fill] (18) at (37.352488,50.000000) {};
    \node [circle,fill] (17) at (23.721567,35.873098) {};
    \node [circle,fill] (16) at (0.000000,50.000000) {};
    \node [circle,fill] (15) at (25.004275,93.304258) {};
    \node [circle,fill] (14) at (50.008550,66.401572) {};
    \node [circle,fill] (13) at (36.394732,22.806567) {};
    \node [circle,fill] (12) at (25.004275,6.695741) {};
    \node [circle,fill] (11) at (74.995722,93.304258) {};
    \node [circle,fill] (10) at (63.622369,77.210533) {};
    \node [circle,fill] (9) at (62.681716,49.982897) {};
    \node [circle,fill] (8) at (50.008550,33.598427) {};
    \node [circle,fill] (7) at (74.995722,6.695741) {};
    \node [circle,fill] (6) at (99.999999,50.000000) {};
    \node [circle,fill] (5) at (76.278432,64.144005) {};
    \node [circle,fill] (4) at (71.609371,49.982897) {};
    \node [circle,fill] (3) at (63.605266,22.772362) {};
    \node [circle,fill] (2) at (88.301692,50.000000) {};
    \node [circle,fill] (1) at (76.261328,35.804686) {};
    \draw [black] (22) to (19);
    \draw [black] (22) to (21);
    \draw [black] (22) to (20);
    \draw [black] (21) to (17);
    \draw [black] (21) to (18);
    \draw [black] (20) to (16);
    \draw [black] (20) to (17);
    \draw [black] (19) to (14);
    \draw [black] (19) to (18);
    \draw [black] (19) to (15);
    \draw [black] (18) to (13);
    \draw [black] (17) to (13);
    \draw [black] (16) to (15);
    \draw [black] (16) to (12);
    \draw [black] (15) to (11);
    \draw [black] (14) to (8);
    \draw [black] (14) to (10);
    \draw [black] (13) to (12);
    \draw [black] (13) to (8);
    \draw [black] (12) to (7);
    \draw [black] (11) to (10);
    \draw [black] (11) to (6);
    \draw [black] (10) to (5);
    \draw [black] (10) to (9);
    \draw [black] (9) to (3);
    \draw [black] (9) to (4);
    \draw [black] (8) to (3);
    \draw [black] (7) to (6);
    \draw [black] (7) to (3);
    \draw [black] (6) to (2);
    \draw [black] (5) to (4);
    \draw [black] (5) to (2);
    \draw [black] (4) to (1);
    \draw [black] (3) to (1);
    \draw [black] (2) to (1);
\end{tikzpicture}}
\qquad
\subfloat[]{\label{fig:planar_3conn_g4_23v}
\begin{tikzpicture}[scale=0.05, inner sep=2pt]
    \definecolor{marked}{rgb}{0.25,0.5,0.25}
    \node [circle,fill] (23) at (59.680318,49.741281) {};
    \node [circle,fill] (22) at (40.305984,49.735726) {};
    \node [circle,fill] (21) at (49.964502,63.229349) {};
    \node [circle,fill] (20) at (49.977392,74.979041) {};
    \node [circle,fill] (19) at (70.727848,60.608079) {};
    \node [circle,fill] (18) at (62.845505,37.113578) {};
    \node [circle,fill] (17) at (49.990609,29.239773) {};
    \node [circle,fill] (16) at (37.141763,37.102531) {};
    \node [circle,fill] (15) at (29.246822,60.632669) {};
    \node [circle,fill] (14) at (36.035715,77.455542) {};
    \node [circle,fill] (13) at (63.921654,77.415821) {};
    \node [circle,fill] (12) at (86.340859,53.267206) {};
    \node [circle,fill] (11) at (75.748795,34.064207) {};
    \node [circle,fill] (10) at (49.979210,16.949781) {};
    \node [circle,fill] (9) at (24.243856,34.050703) {};
    \node [circle,fill] (8) at (13.645776,53.260846) {};
    \node [circle,fill] (7) at (49.993745,97.552012) {};
    \node [circle,fill] (6) at (99.999999,61.220342) {};
    \node [circle,fill] (5) at (71.545746,14.035308) {};
    \node [circle,fill] (4) at (28.433960,14.041851) {};
    \node [circle,fill] (3) at (0.000000,61.225610) {};
    \node [circle,fill] (2) at (80.900456,2.447988) {};
    \node [circle,fill] (1) at (19.098960,2.449300) {};
    \draw [black] (23) to (18);
    \draw [black] (23) to (22);
    \draw [black] (23) to (19);
    \draw [black] (22) to (15);
    \draw [black] (22) to (16);
    \draw [black] (21) to (15);
    \draw [black] (21) to (20);
    \draw [black] (21) to (19);
    \draw [black] (20) to (13);
    \draw [black] (20) to (14);
    \draw [black] (19) to (12);
    \draw [black] (19) to (13);
    \draw [black] (18) to (17);
    \draw [black] (18) to (11);
    \draw [black] (17) to (10);
    \draw [black] (17) to (16);
    \draw [black] (16) to (9);
    \draw [black] (15) to (8);
    \draw [black] (15) to (14);
    \draw [black] (14) to (7);
    \draw [black] (13) to (7);
    \draw [black] (12) to (11);
    \draw [black] (12) to (6);
    \draw [black] (11) to (5);
    \draw [black] (10) to (4);
    \draw [black] (10) to (5);
    \draw [black] (9) to (8);
    \draw [black] (9) to (4);
    \draw [black] (8) to (3);
    \draw [black] (7) to (6);
    \draw [black] (7) to (3);
    \draw [black] (6) to (2);
    \draw [black] (5) to (2);
    \draw [black] (4) to (1);
    \draw [black] (3) to (1);
    \draw [black] (2) to (1);
\end{tikzpicture}}
    \caption{A planar 3-connected platypus on 22 and 23 vertices, respectively.}

\end{figure}







%

\section*{Acknowledgements}
Most computations for this work were carried out using the Stevin Supercomputer Infrastructure at Ghent University. We would also like to thank Gunnar Brinkmann for his advice on \verb|plantri|.


\section*{References}


\end{document}